\documentclass[twoside, 12pt]{amsart}
\usepackage{latexsym}
\usepackage{amssymb,amsmath,amsopn}
\usepackage[dvips]{graphicx}   %To insert ps figures
\usepackage{color,epsfig}      %To insert ps figures 

\newtheorem{theorem}{Theorem}

\newtheorem{defi}{Definition}
\newtheorem{problem}{Problem}
\newtheorem{cor}{Corollary}
\newtheorem{conj}{Conjecture}

\DeclareMathOperator{\conv}{conv}
\DeclareMathOperator{\dif}{d}
\DeclareMathOperator{\vol}{vol}
\DeclareMathOperator{\area}{area}

\DeclareMathOperator{\bd}{bd}

\newcommand{\K}{\mathcal{K}}
\renewcommand{\S}{\mathcal{S}}
\renewcommand{\Re}{\mathbb R}
\newcommand{\B}{\mathbf B}
\newcommand{\Sph}{\mathbb{S}}
\renewcommand{\P}{\mathcal{P}}

\begin{document}

\title[On the volume of convex hull]{On the volume of the convex hull of two convex bodies}
\author[\'A. G.Horv\'ath and Z. L\'angi]{\'Akos G.Horv\'ath and Zsolt L\'angi}
%\date{2013 Jan}

\address{\'A. G.Horv\'ath, Dept. of Geometry, Budapest University of Technology,
Egry J\'ozsef u. 1., Budapest, Hungary, 1111}
\email{ghorvath@math.bme.hu}
\address{Zsolt L\'angi, Dept. of Geometry, Budapest University of Technology,
Egry J\'ozsef u. 1., Budapest, Hungary, 1111}
\email{zlangi@math.bme.hu}

\subjclass{52A40, 52A38, 26B15}
\keywords{convex hull, volume inequality, isometry, reflection, translation.}

\begin{abstract}
In this note we examine the volume of the convex hull of two congruent copies of a convex body in Euclidean $n$-space,
under some subsets of the isometry group of the space.
We prove inequalities for this volume if the two bodies are translates, or reflected copies of each other about a common point
or a hyperplane containing it. In particular, we give a proof of a related conjecture of Rogers and Shephard.
\end{abstract}

\maketitle

\section{Introduction}

The volume of the convex hull of two convex bodies in the Euclidean $n$-space $\Re^n$ has been in the focus of research since the 1950s.
One of the first results in this area is due to F\'ary and R\'edei \cite{fary-redei},
who proved that if one of the bodies is translated on a line at a constant velocity, then the volume of their convex hull is a convex function of time.
This result was reproved by Rogers and Shephard \cite{rogers-shephard2} in 1958, using a more general theorem about the so-called \emph{linear parameter systems},
and for polytopes by Ahn, Brass and Shin \cite{ahn} in 2008.

In this paper we investigate the following quantities.

\begin{defi}
For two convex bodies $K$ and $L$ in $\Re^n$, let
\[
c(K,L)=\max\left\{ \vol (\conv (K'\cup L')) : K' \cong K, L' \cong L \mbox{ and } K'\cap L'\neq\emptyset \right\},
\]
where $\cong$ and $\vol$ denotes congruence and $n$-dimensional Lebesgue measure, respectively. Furthermore, if $\S$ is a set of isometries of $\Re^n$, we set
\[
c(K|\S)= \frac{\max \left\{ \vol( \conv (K \cup K' )) : K \cap K' \neq \emptyset , K' = \sigma(K) \hbox{ for some } \sigma \in \S \right\}}{\vol(K)}.
\] 
\end{defi}

We note that a quantity similar to $c(K,L)$ was defined by Rogers and Shephard \cite{rogers-shephard2}, in which congruent copies were replaced
by translates. Another related quantity is investigated in \cite{gho}, where the author examines $c(K,K)$ in the special case that $K$ is a regular simplex and the two congruent copies have the same centre.

In \cite{rogers-shephard2}, Rogers and Shephard used linear parameter systems to show that the minimum of $c(K|\S)$,
taken over the family of convex bodies in $\Re^n$, is its value for an $n$-dimensional Euclidean ball,
if $\S$ is the set of translations or that of reflections about a point.
Nevertheless, their method, approaching a Euclidean ball by suitable Steiner symmetrizations and showing that during this process the examined quantities do not increase, does not characterize the convex bodies for which the minimum is attained;
they conjectured that, in both cases, the minimum is attained only for ellipsoids (cf. p. 94 of \cite{rogers-shephard2}).
We note that the method of Rogers and Shephard \cite{rogers-shephard2} was used also in \cite{macbeath}.
We remark that the conjecture in \cite{rogers-shephard2} follows from a straightforward modification of Theorems 9 and 10 of \cite{MM06}.
This proof requires an extensive knowledge of measures in normed spaces. Our goal in part is to give a proof using more classical tools.

We treat these problems in a more general setting. For this purpose, let $c_i(K)$ be the value of $c(K|\S)$, where $\S$ is the set of reflections about the $i$-flats of $\Re^n$, and $i=0,1,\ldots,n-1$.
Similarly, let $c^{tr}(K)$ and $c^{co}(K)$ be the value of $c(K|\S)$ if $\S$ is the set of translations and that of all the isometries, respectively.
In Section~\ref{sec:standard} we examine the minima of these quantities. In particular, in Theorem~\ref{thm:translate}, we give another proof that the minimum of $c^{tr}(K)$, over the family of convex bodies in $\Re^n$, is its value for Euclidean balls, and show also that the minimum is attained if, and only if, $K$ is an ellipsoid. This verifies the conjecture in \cite{rogers-shephard2} for translates.
In Theorem~\ref{thm:pointreflection}, we characterize the plane convex bodies for which $c^{tr}(K)$ is attained for any touching pair of translates of $K$, showing a connection of the problem with Radon norms.
In Theorems~\ref{thm:centralsymmetry} and \ref{thm:hyperplane}, we present similar results about the minima of $c_0(K)$ and $c_{n-1}(K)$, respectively. In particular, we prove that, over the family of convex bodies, $c_0(K)$ is minimal for ellipsoids, and $c_{n-1}(K)$ is minimal for Euclidean balls. The first result proves the conjecture of Rogers and Shephard for copies reflected about a point.

The maximal values of $c^{tr}(K)$ and $c_0(K)$, for $K \in \K_n$, and the convex bodies for which these values are attained, are determined in \cite{rogers-shephard2}; the authors prove that $c_0(K) \leq 2^n$ with equality (only) for simplices, and 
$c^{tr}(K) \leq n+1$, with equality for what the authors call \emph{pseudo-double-pyramids}.
Using a suitable simplex as $K$, it is easy to see that the set $\{c_i(K) : K \in \K_n \}$ is not bounded from above for
$i = 1,\ldots,n-1$. This readily yields the same statement for $c^{co}(K)$ as well. 

In Section~\ref{sec:discrete} we introduce variants of these quantities for convex $m$-gons in $\Re^2$, and for small values of $m$, characterize the polygons for which these quantities are minimal. Finally, in Section~\ref{sec:remarks} we collect some additional remarks and questions.

During the investigation, $\K_n$ denotes the family of $n$-dimensional convex bodies. We let $\B^n$ be the $n$-dimensional unit ball with the origin $o$
of $\Re^n$ as its centre, and set $\Sph^{n-1} = \bd \B^n$ and $v_n = \vol(\B^n)$. Finally, we denote $2$- and $(n-1)$-dimensional Lebesgue measure by $\area$ and $\vol_{n-1}$, respectively. For any $K \in \K_n$ and $u \in \Sph^{n-1}$, $K | u^\perp$ denotes the orthogonal projection of $K$ onto the hyperplane passing through the origin $o$ and perpendicular to $u$. The \emph{polar} of a convex body $K$, containing $o$ in its interior, is the set
\[
K^\circ = \{ v \in \Re^n : \langle u, v \rangle \leq 1 \hbox{ for every } u \in K \}, 
\]
where $\langle .,. \rangle$ is the usual inner product of $\Re^n$.

\section{The minima of $c^{tr}(K)$, $c_0(K)$ and $c_{n-1}(K)$}\label{sec:standard}

\begin{theorem}\label{thm:translate}
For any $K\in \K_n$ with $n\geq 2$, we have $c^{tr}(K) \geq 1 + \frac{2v_{n-1}}{v_n}$ with equality if, and only if, $K$ is an ellipsoid.
\end{theorem}

\begin{proof}
Since for ellipsoids $c^{tr}(K) = 1 + \frac{2v_{n-1}}{v_n}$,
it suffices to show that if $c^{tr}(K) \leq 1 + \frac{2v_{n-1}}{v_n}$, then $K$ is an ellipsoid.

Let $K \in \K_n$ be a convex body such that $c^{tr}(K) \leq 1 + \frac{2v_{n-1}}{v_n}$. 
Consider the case that $K$ is not centrally symmetric.
Let $\sigma : \K_n \to \K_n$ be a Steiner symmetrization about any hyperplane.
Then Lemma 2 of \cite{rogers-shephard2} yields that $c^{tr}(K) \geq c^{tr}(\sigma(K))$.
On the other hand, Lemma 10 of \cite{MM06} states that, for any not centrally symmetric convex body, there is an orthonormal basis
such that subsequent Steiner symmetrizations, through hyperplanes perpendicular to its vectors, yields a centrally symmetric 
convex body, different from ellipsoids.
Combining these statements, we obtain that there is an $o$-symmetric convex body $K' \in \K_n$ that is not an ellipsoid and satisfies
$c^{tr}(K) \geq c^{tr}(K')$.
Thus, it suffices to prove the assertion in the case that $K$ is centrally symmetric.

Assume that $K$ is $o$-symmetric, and that $c^{tr}(K) \leq 1 + \frac{2v_{n-1}}{v_n}$.
For any $u \in \Sph^{n-1}$, let $d_K(u)$ denote the length of a maximal chord parallel to $u \in \Sph^{n-1}$.
Observe that for any such $u$, $K$ and $d_K(u) u + K$ touch each other and
\begin{equation}\label{eq:translate}
\frac{\vol (\conv (K \cup (d_K(u) u + K)))}{\vol(K)} = 1 + \frac{d_K(u) \vol_{n-1}(K|u^\perp)}{\vol(K)}.
\end{equation}
Clearly, $c^{tr}(K)$ is the maximum of this quantity over $u \in \Sph^{n-1}$.

Let $u \mapsto r_K(u) = \frac{d_K(u)}{2}$ be the radial function of $K$.
From (\ref{eq:translate}) and the inequality $c^{tr}(K) \leq 1 + \frac{2v_{n-1}}{v_n}$, we obtain that for any $u \in \Sph^{n-1}$
\begin{equation}\label{eq:inequality}
\frac{v_{n-1} \vol(K)}{v_n \vol_{n-1}(K|u^\perp)} \geq r_K(u).
\end{equation}
Applying this for the polar form of the volume of $K$, we obtain
\[
\vol(K) = \frac{1}{n} \int\limits_{\Sph^{n-1}} \left( r_K(u) \right)^n \dif u \leq \frac{1}{n} \frac{v_{n-1}^n}{v_n^n} \left( \vol(K) \right)^n \int\limits_{\Sph^{n-1}} \frac{1}{\left( \vol_{n-1}(K|u^\perp) \right)^n} \dif u ,
\]
which yields
\begin{equation}\label{eq:oneside}
\frac{v_n^n n}{v_{n-1}^n \left( \vol(K) \right)^{n-1}} \leq \int\limits_{\Sph^{n-1}} \frac{1}{\left( \vol_{n-1}(K|u^\perp) \right)^n} \dif u
\end{equation}

On the other hand, combining Cauchy's surface area formula with Petty's projection inequality, we obtain that for every $p \geq -n$,
\[
v_n^{1/n}\left( \vol(K) \right)^{\frac{n-1}{n}} \leq v_n\left(\frac{1}{nv_{n}}\int\limits_{S^{n-1}}\left(\frac{\vol_{n-1}(K|u^\perp)}{v_{n-1}}\right)^p \dif u \right)^{\frac{1}{p}},
\]
with equality only for Euclidean balls if $p > -n$, and for ellipsoids if $p=-n$ (cf. e.g. Theorems 9.3.1 and 9.3.2 in \cite{gardner}).

This inequality, with $p=-n$ and after some algebraic transformations, implies that
\begin{equation}\label{eq:otherside}
\int\limits_{\Sph^{n-1}} \frac{1}{\left( \vol_{n-1}(K|u^\perp) \right)^n} \dif u \leq \frac{v_n^n n}{v_{n-1}^n \left( \vol(K) \right)^{n-1}}
\end{equation}
with equality if, and only if, $K$ is an ellipsoid.
Combining (\ref{eq:oneside}) and (\ref{eq:otherside}), we can immediately see that if $c^{tr}(K)$ is minimal, then $K$ is an ellipsoid, and
in this case $c^{tr}(K) = 1 + \frac{2v_{n-1}}{v_n}$.
\end{proof}

If, for a convex body $K \in \K_n$, we have that $\vol (\conv ((v+K) \cup (w+K)))$ has the same value for any touching pair of translates,
let us say that $K$ satisfies the \emph{translative constant volume property}.
In the next part of Section~\ref{sec:standard}, we characterize the plane convex bodies with this property.
Before doing this, we recall that a $2$-dimensional $o$-symmetric convex curve is a Radon curve, if, for the convex hull $K$ of 
a suitable affine image of the curve, it holds that $K^\circ$ is a rotated copy of $K$ by $\frac{\pi}{2}$ (cf. \cite{MS06}).
Furthermore, a norm is a \emph{Radon norm} if the boundary of its unit disk is a Radon curve.

\begin{theorem}\label{thm:pointreflection}
For any plane convex body $K \in \K_2$ the following are equivalent.
\begin{itemize}
\item[(1)] $K$ satisfies the translative constant volume property.
\item[(2)] The boundary of $\frac{1}{2}(K-K)$ is a Radon curve.
\item[(3)] $K$ is a body of constant width in a Radon norm.
\end{itemize}
\end{theorem}

\begin{proof}
Recall that a convex body $K$ is a body of constant width in a normed space with unit ball $M$ if, and only if, its central symmetral $\frac{1}{2}(K-K)$ is a homothetic copy of $M$. Thus, (2) and (3) are clearly equivalent, and we need only show that (1) and (2) are.

Let $K \in \K_2$. For any $u \neq o$, let $w_K(u)$ denote the width of $K$ in the direction of $u$.
Then, using the notation $u = w-v$, for any touching pair of translates, we have
\begin{equation}\label{eq:plane_translate}
\area( \conv ((v+K) \cup (w+K))) = \area(K) + d_K(u) w_K(u^\perp).
\end{equation}

Since for any direction $u$, we have $d_K(u) = d_{\frac{1}{2}(K-K)}(u)$ and $w_K(u) = w_{\frac{1}{2}(K-K)}(u)$, $K$ satisfies the translative constant volume property if, and only if, its central symmetral does. Thus, we may assume that $K$ is $o$-symmetric.
Now let $x \in \bd K$. Then the boundary of $\conv(K \cup (2x+K))$ consists of an arc of $\bd K$,
its reflection about $x$, and two parallel segments, each contained in one of the two common supporting lines of $K$ and $2x+K$,
which are parallel to $x$.
For some point $y$ on one of these two segments, set $A_K(x) = \area \conv \{ o,x,y\}$ (cf. Figure~\ref{fig:radon}).
Clearly, $A_K(x)$ is independent of the choice of $y$.
Then we have for every $x \in \bd K$, that $d_K(x) w_K(x^\perp) = 8 A_K(x)$.

\begin{figure}[here]
\includegraphics[width=0.5\textwidth]{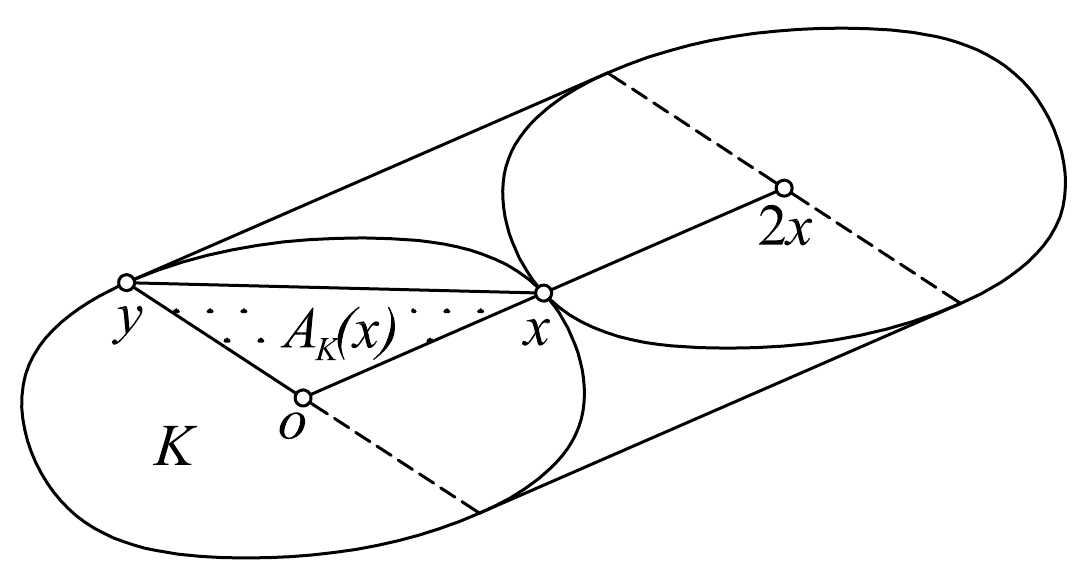}
\caption[]{An illustration for the proof of Theorem~\ref{thm:pointreflection}}
\label{fig:radon}
\end{figure}

Assume that $A_K(x)$ is independent of $x$. We need to show that in this case $\bd K$ is a Radon curve.
It is known (cf. \cite{MS06}), that $\bd K$ is a Radon curve if, and only if, 
in the norm of $K$, \emph{Birkhoff-orthogonality} is a symmetric relation.
Recall that in a normed plane with unit ball $K$, a vector $x$ is called \emph{Birkhoff-orthogonal} to a vector $y$, denoted by $x \perp_B y$,
if $x$ is parallel to a line supporting $|| y || \bd K$ at $y$ (cf. \cite{AB89}).

Observe that for any $x, y \in \bd K$, $x \perp_B y$ if, and only if, $A_K(x) = \area(\conv \{ o,x,y\})$, or in other words, if
$\area(\conv\{o,x,y\})$ is maximal over $y \in K$.
Clearly, it suffices to prove the symmetry of Birkhoff orthogonality for $x,y \in \bd K$.
Consider a sequence $x \perp_B y \perp_B z$ for some $x, y, z \in \bd K$.
Then we have $A_K(x) = \area \conv \{ o,x,y \}$ and $A_K(y) = \area(\conv \{ o,y,z\})$.
By the maximality of $\area(\conv \{ o,y,z\})$, we have $A_K(x) \leq A_K(y)$ with equality if, and only if, $y \perp_B x$.
This readily implies that Birkhoff orthogonality is symmetric, and thus, that $\bd K$ is a Radon curve.
The opposite direction follows from the definition of Radon curves and polar sets.
\end{proof}

\begin{remark}\label{rem:translate_plane}
The proof of Theorem~\ref{thm:pointreflection} can be modified to prove Theorem~\ref{thm:translate} in the plane.
\end{remark}

We sketch this proof. We note that a simplified version of this argument can be applied for Theorem~\ref{thm:hyperplane} in the planar case.

\begin{proof}
Using (\ref{eq:plane_translate}), we obtain that $c^{tr}(K) = 1 + \frac{\max \{ d_K(u) w_K(u^\perp) : u \in \Sph^{n-1}\} }{\area(K)}$.
Note that the numerator in this expression is the same for $\frac{1}{2}(K-K)$ as for $K$.
By the Brunn-Minkowski Inequality, $\area(K) \leq \area\left(\frac{1}{2}(K-K) \right)$, with equality if, and only if $K$ is centrally symmetric, and thus,
it suffices to prove the assertion under the assumption that $K$ is $o$-symmetric.

An argument similar to the one in the proof of Theorem~\ref{thm:pointreflection} yields that there is a Radon curve $g$ such that
$K \subseteq K' = \conv g$ and $\max \{ d_K(u) w_K(u^\perp) : u \in \Sph^{n-1}\} = \max \{ d_{K'}(u) w_{K'}(u^\perp) : u \in \Sph^{n-1}\}$.
This implies that $c^{tr}(K') \leq c^{tr}(K)$, with equality if, and only if $K=K'$, and thus, we may assume that $\bd K$ is a Radon curve.
Since $c^{tr}(K)$ is affine invariant, we may also assume that $K^\circ$ is the rotated copy of $K$ by $\frac{\pi}{2}$; in this case
$d_K(u) w_K(u^\perp) = 4$ for any $u \in \Sph^{n-1}$.

Finally, from the Blaschke-Santal\'o inequality (cf. \cite{gardner}), we have
\[
\left( \area(K) \right)^2 = \area(K) \area(K^\circ) \leq v_2^2,
\]
with equality if, and only if, $K$ is an ellipse. Thus, $\area(K) \leq v_2$, from which the assertion readily follows.
\end{proof}

The following theorem is an immediate consequence of Lemma 10 of \cite{MM06} and Theorem~\ref{thm:translate}.

\begin{theorem}\label{thm:centralsymmetry}
For any $K \in \K_n$ with $n \geq 2$, $c_0(K) \geq 1 + \frac{2v_{n-1}}{v_n}$, with equality if, and only if, $K$ is an ellipsoid.
\end{theorem}

Our next result shows an inequality for $c_{n-1}(K)$.

\begin{theorem}\label{thm:hyperplane}
For any $K \in \K_n$ with $n \geq 2$, $c_{n-1}(K) \geq 1+\frac{2v_{n-1}}{v_n}$, with equality if, and only if, $K$ is a Euclidean ball.
\end{theorem}

\begin{proof}
For a hyperplane $\sigma \subset \Re^n$, let $K_\sigma$ denote the reflected copy of $K$ about $\sigma$.
Furthermore, if $\sigma$ is a supporting hyperplane of $K$, let $K_{-\sigma}$ be the reflected copy of $K$ about the other supporting hyperplane
of $K$ parallel to $\sigma$.
Clearly,
\[
c_{n-1}(K) = \frac{1}{\vol(K)} \max \{ \vol ( \conv (K \cup K_\sigma)) : \sigma \hbox{ is a supporting hyperplane of } K \}.
\]

For any direction $u \in\Sph^{n-1}$, let $H_K(u)$ be the right cylinder circumscribed about $K$ and with generators parallel to $u$.
Observe that for any $u \in \Sph^{n-1}$ and supporting hyperplane $\sigma$ perpendicular to $u$, we have
$\vol ( \conv (K \cup K_\sigma)) + \vol ( \conv (K \cup K_{-\sigma}) = 2\vol(K) + 2\vol(H_K(u)) = 2\vol(K) + 2 w_K(u) \vol_{n-1}(K|u^\perp)$.
Thus, for any $K \in \K_n$,
\begin{equation}\label{eq:hypcylinder}
c_{n-1}(K) \geq 1 + \frac{\max \{ w_K(u) \vol_{n-1}(K|u^\perp) : u \in \Sph^{n-1} \} }{\vol(K)} .
\end{equation}

Let $d_K(u)$ denote the length of a longest chord of $K$ parallel to $u \in \Sph^{n-1}$.
Observe that for any $u \in \Sph^{n-1}$, $d_K(u) \leq w_K(u)$, and thus for any convex body $K$,
\[
c_{n-1}(K) \geq c^{tr}(K).
\]
This readily implies that $c_{n-1}(K) \geq 1 + \frac{2v_{n-1}}{v_n}$, and if here there is equality for some $K \in \K_n$, then $K$ is an ellipsoid.
On the other hand, in case of equality, for any $u \in \Sph^{n-1}$ we have $d_K(u) = w_K(u)$, which yields that $K$ is a Euclidean ball.
This finishes the proof of the theorem.
\end{proof}

\section{Discrete versions of the problems in $\Re^2$}\label{sec:discrete}

In this section, let $\P_m$ denote the family of convex $m$-gons in the plane $\Re^2$.
It is a natural question to ask about the minima of the quantities defined in the introduction over $\P_m$.
More specifically, we set
\begin{eqnarray*}
t_m & = & \min \{ c^{tr}(P) : P \in \P_m \};\\
p_m & = & \min \{ c_0(P) : P \in \P_m\};\\
l_m & = & \min \{ c_1(P) : P \in \P_m \}.
\end{eqnarray*}

\begin{theorem}\label{thm:discrete}
We have the following.
\begin{itemize}
\item[(1)] $t_3 = t_4 = 3$ and $t_5 = \frac{10+\sqrt{5}}{5}$. Furthermore, $c^{tr}(P) = 3$ holds for any triangle and quadrilateral, and if $c^{tr}(P) = t_5$ for some $P \in \P_5$, then $P$ is an affine regular pentagon.
\item[(2)] $p_3 = 4$, $p_4 = 3$ and $p_5 = 2 + \frac{4 sin \frac{\pi}{5}}{5}$. Furthermore, in each case, the minimum is attained only for affinely regular polygons.
\item[(3)] $l_3 = 4$ and $l_4 =3$.  Furthermore, among triangles, the minimum is attained only for regular ones, and among quadrilaterals for rhombi.
\end{itemize}
\end{theorem}

\begin{proof}[Proof of (1)]
It suffices to examine the case that the intersection of the two polygons is a vertex of both.
It is fairly elementary to show that for any triangle and quadrilateral $T$ we have $c^{tr}(T) = 3$. This implies also $t_3=t_4=3$.

Consider a convex pentagon $P$ with vertices $a_i$, $i=1,2,\ldots,5$ in counterclockwise order.
Assume, without loss of generality, that $\area (\conv \{ a_1,a_3,a_4\}) \leq \area (\conv \{ a_1,a_3,a_5\})$.
Observe that in this case $\area(\conv \{ P \cup (a_3-a_1+P)\}) = 3 \area(P) - 2\area(\conv \{a_3,a_4,a_5\})$ (cf. Figure~\ref{fig:translate}).
Repeating this argument for any $a_{i+2}-a_i+P$, we obtain that
\begin{equation}\label{eq:translates_pentagons}
3 - \frac{2\min \{ \area(\conv \{ a_{i-1},a_i,a_{i+1}\}) : i=1,2,\ldots,5 \}}{\area(P)} \leq c^{tr}(P).
\end{equation}

\begin{figure}[here]
\includegraphics[width=0.4\textwidth]{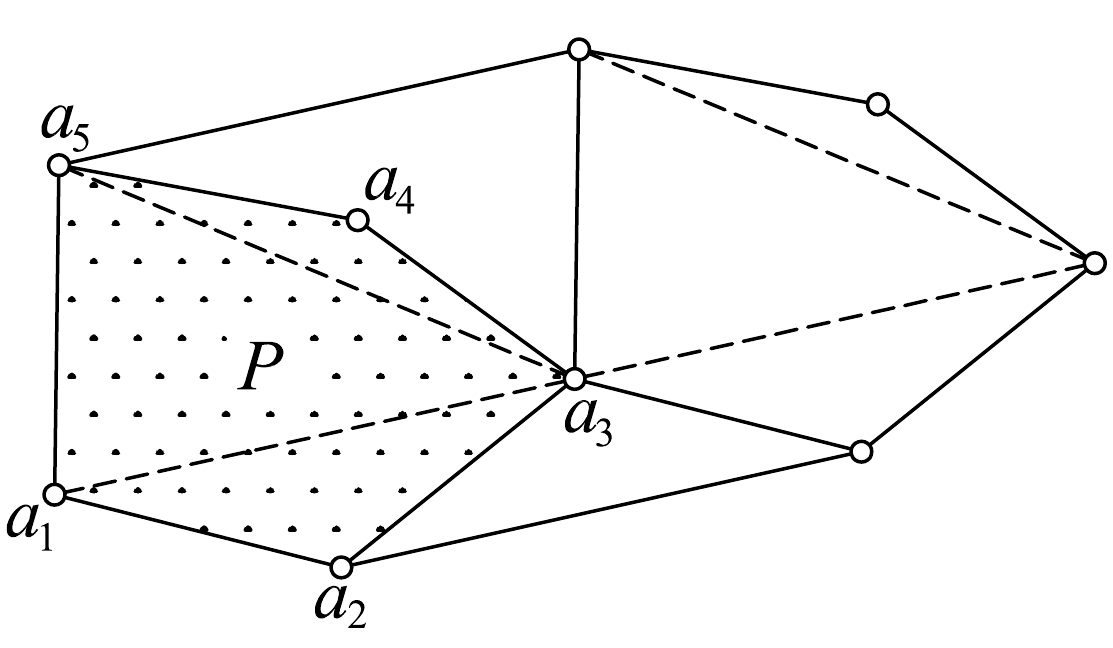}
\caption[]{An illustration for the proof of (1) of Theorem~\ref{thm:discrete}}
\label{fig:translate}
\end{figure}

On the other hand, from \cite{GL08} it follows that, among pentagons, the left-hand side is minimal if, and only if, $P$ is an affine regular pentagon.
Since for any such pentagon the two sides of (\ref{eq:translates_pentagons}) are equal, the assertion readily follows.
\end{proof}

\begin{proof}[Proof of (2)]
For triangles, the statement is trivial and for quadrilaterals it is a simplified version of the one for pentagons.
Hence, we prove only the last case.
Let $P$ be a pentagon such that $c_1(P)$ is minimal, with vertices $a_1,a_2,\ldots, a_5$ in this counterclockwise order.
Since for a regular pentagon $\bar{P}$, $c_1(\bar{P}) = 2 + \frac{4 sin \frac{\pi}{5}}{5} \approx 2.47$,
we may assume that $c_1(P)$ is not less than this quantity, which we denote by $C$.
It suffices to deal with the case that $P$ is reflected about one of its vertices.
For $i=1,2,\ldots, 5$, set $A_i = \area(\conv \{ a_{i-1},a_i,a_{i+1}\} )$.

\emph{Case 1}, $\conv(P \cup (2z-P))$ is a quadrilateral for some vertex $z$ of $P$.
Without loss of generality, we may assume that $z$ is the origin, and, since $c_1(K)$ is invariant under affine transformations,
that this quadrilateral is a unit square. Let $\conv(P \cup (-P)) = \conv\{ a_i, a_{i+1}, -a_i, -a_{i+1}\}$.
Now, observe that $\conv (P \cup (2a_i - P))$ contains two triangles of area $\frac{1}{6}$ that do not overlap $P \cup (2a_i - P)$ (cf. Figure~\ref{fig:reflection}).
Since we clearly have $\area(P) \leq \frac{1}{2}$, this immediately yields that $\area(\conv (P \cup (2a_i-P))) \geq \frac{2\area(P)+\frac{1}{3}}{\area(P)} \geq \frac{8}{3} > C \geq c_1(P)$, a contradiction.

\begin{figure}[here]
\includegraphics[width=0.35\textwidth]{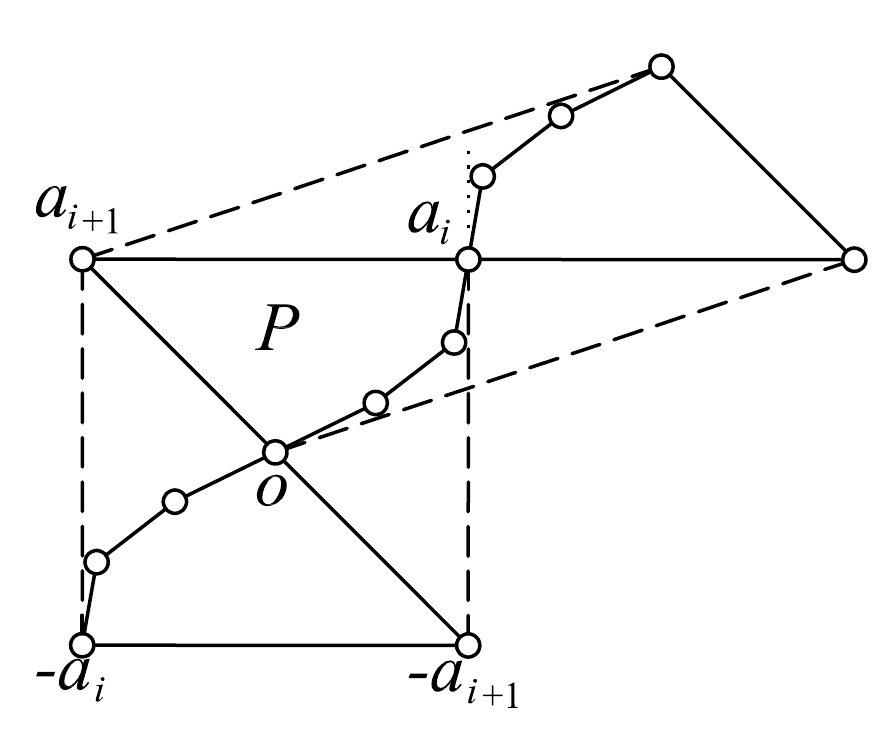}
\caption[]{An illustration for the proof of (2) of Theorem~\ref{thm:discrete}}
\label{fig:reflection}
\end{figure}

\emph{Case 2}, $\conv(P \cup (2z-P))$ is a hexagon for some vertex $z$.
We label the vertices of $P$ in such a way that the vertices of $\conv(P \cup (2z-P))$ are $a_1, a_2, a_3$ and their reflections about $z=a_5$.
Like in Case 1, we may assume that $a_5$ is the origin, and that $a_1,a_3,-a_1,-a_3$ are the vertices of a unit square.
Note that $\area (P) = A_2 + A_4 + \area(\conv \{ a_1,a_3,a_5\}) = A_2 + A_4 + \frac{1}{4}$.
Applying for $i=2$, $i=4$ and $i=5$ the assumption that
\begin{equation}\label{eq:assumption}
\area (\conv \{ P \cup (2a_i-P)\}) \leq C \area(P) < \frac{5}{2} \area (P)
\end{equation}
for every $i$, we obtain that
\begin{equation}\label{eq:1}
3A_2 < A_4 + \frac{1}{4}, \quad 3A_4 < A_2 + \frac{1}{4} \quad \hbox{and} \quad \frac{3}{4} < A_2 + 5A_4.
\end{equation}
On the other hand, this inequality system has no solution.

\emph{Case 3}, $\conv(P \cup (2a_i-P))$ is an octagon for every value of $i$.
Then $c_1(P) = 2 + \frac{2\max \{ A_i : i=1,2,\ldots,5\} }{\area (P)}$.
We show that if $c_1(P)$ is minimal, then $A_i = A_j$ for every $i$ and $j$.
Suppose for contradiction that $A_{i-1} < A_{i+1}$ and that $A_{i+1}$ is maximal for $i$.
Then, by moving $a_i$ parallel to $[a_{i-1},a_{i+1}]$ a little towards $a_{i+1}$, we increase $A_{i-1}$, decrease $A_{i+1}$,
and do not change $\area(P)$ and the rest of the $A_j$s. Thus, decreasing the number of the maxima of the $A_j$s,
we may decrease $c_1(P)$ in at most four steps; a contradiction.
Hence, we may assume that $A_i$ is the same value for every $i$.
On the other hand, it is known that this property characterizes affine regular pentagons (cf. \cite{GL08}).
\end{proof}

\begin{proof}[Proof of (3)]
Let $T$ be a triangle with vertices $a_1$, $a_2$ and $a_3$ in counterclockwise order. Let $\alpha_i$ and $t_i$ be the angle of $T$ at $a_i$ and
the length of the side opposite of $a_i$, respectively.
Consider a line $L$ through $a_1$ that does not cross $T$, and  let $\beta_1$ and $\gamma_1$ be the oriented angles from $L$ to $[a_1,a_2]$,
and from $[a_1,a_3]$ to $L$, respectively. Let $a_2'$ and $a_3'$ be the orthogonal projections of $a_2$ and $a_3$ on $L$, respectively, and set
$A_1(L) = \area(\conv \{ a_1,a_2,a_2' \} )+ \area(\conv \{ a_1,a_3,a_3'\})$.
By elementary calculus, it is easy to see that, among the lines through $a_1$, the one maximizing $A_1(L)$ satisfies $\beta_1 = \gamma_1 = \frac{\pi}{4}$
if $\alpha_1 = \frac{\pi}{2}$, and $t_2^2 \cos 2 \gamma_1 = t_3^2 \cos 2 \beta_1$ otherwise.
This yields, in particular, that for the line maximizing $A_1(L)$, which we denote by $L_1$, we have that $\beta_1$ and $\gamma_1$ are acute.
We define $\beta_i$, $\gamma_i$, $A_i(L)$ and $L_i$ for $i=2,3$ similarly.

By elementary computations, we have that if $\alpha_i \neq \frac{\pi}{2}$, then
\[
\frac{A_i(L_i)}{\area(T)} = \frac{t_{i+1}^2 \sin 2 \gamma_i + t_{i-1}^2 \sin 2 \beta_i}{2t_2t_3 \sin \alpha_i} = \frac{|\cos \alpha_i|}{\sqrt{\cos2\beta_i \cos 2\gamma_i}} 
\]
Since the function $x \mapsto \log \cos x$ is strictly concave on $\left( 0, \frac{\pi}{2} \right)$ and $\left( \frac{\pi}{2}, \pi \right)$,
we have that
\[
\frac{A_i(L_i)}{\area(T)} \geq \frac{|\cos \alpha_i|}{\sqrt{\cos^2 \alpha_i }} = 1, 
\]
with equality if, and only if $\beta_i = \gamma_i$; that is, if $t_{i-1} = t_{i+1}$.
This readily implies that
\[
c_1(T) = 2 + \frac{2 \max \{ A_1(L_1), A_2(L_2), A_3(L_3) \}}{\area(T)} \geq 4,
\]
with equality if, and only if, $T$ is equilateral.
For quadrilaterals, a similar argument yields the assertion.
\end{proof}

\section{Remarks and questions}\label{sec:remarks}

We start with a conjecture.

\begin{conj}
Let $n \geq 3$ and $1 < i < n-1$. Prove that, for any $K \in \K_n$, $c_i(K) \geq 1 + \frac{2v_{n-1}}{v_n}$. Is it true that equality holds only for Euclidean balls?
\end{conj}

From Theorem~\ref{thm:hyperplane} we obtain the following.

\begin{remark}
For any $K \in \K_n$ with $n \geq 2$, we have $c^{co}(K) \geq 1+\frac{2v_{n-1}}{v_n}$, with equality if, and only if, $K$ is a Euclidean ball.
\end{remark}

In Theorem~\ref{thm:pointreflection}, we proved that in the plane, a convex body satisfies the translative constant volume property if, and only if, it is of constant width in a Radon plane. It is known (cf. \cite{AB89} or \cite{MS06}) that for $n \geq 3$, if every planar section of a normed space is Radon, then the space is Euclidean; that is, its unit ball is an ellipsoid. It is known that there are different convex bodies with the same width and brightness functions, and thus, characterizing the convex bodies satisfying the translative constant volume property seems difficult.
Nevertheless, for centrally symmetric bodies the following seems plausible.

\begin{conj}
Let $n \geq 3$. If some $o$-symmetric convex body $K \in \K_n$ satisfies the translative constant volume property, then $K$ is an ellipsoid.
\end{conj}

Furthermore, we remark that the proof of Theorem~\ref{thm:pointreflection} can be extended, using the Blaschke-Santal\'o inequality, to prove Theorems~\ref{thm:translate} and \ref{thm:centralsymmetry} in the plane.
Similarly, Theorem~\ref{thm:hyperplane} can be proven by a modification of the proof of Theorem~\ref{thm:translate}, in which we estimate the volume of the polar body using the width function of the original one, and apply the Blaschke-Santal\'o inequality.

Like in \cite{rogers-shephard2}, Theorems~\ref{thm:translate} and \ref{thm:hyperplane} yield information about circumscribed cylinders.
Note that the second corollary is a strenghtened version of Theorem 5 in \cite{rogers-shephard2}.

\begin{cor}\label{cor:rightcylinder}
For any convex body $K \in \K_n$, there is a direction $u \in \Sph^{n-1}$ such that the right cylinder $H_K(u)$, circumscribed about $K$
and with generators parallel to $u$ has volume
\begin{equation}\label{eq:rightcylinder}
\vol(H_K(u)) \geq \left( 1 + \frac{2v_{n-1}}{v_n} \right) \vol(K).
\end{equation}
Furthermore, if $K$ is not a Euclidean ball, then the inequality sign in (\ref{eq:rightcylinder}) is a strict inequality.
\end{cor}

\begin{cor}\label{cor:eachcylinder}
For any convex body $K \in \K_n$, there is a direction $u \in \Sph^{n-1}$ such that any cylinder $H_K(u)$, circumscribed about $K$
and with generators parallel to $u$, has volume
\begin{equation}\label{eq:eachcylinder}
\vol(H_K(u)) \geq \left( 1 + \frac{2v_{n-1}}{v_n} \right) \vol(K).
\end{equation}
Furthermore, if $K$ is not an ellipsoid, then the inequality sign in (\ref{eq:eachcylinder}) is a strict inequality.
\end{cor}

Let $P_m$ be a regular $m$-gon in $\Re^2$. We ask the following.

\begin{problem}
Prove or disprove that for any $m \geq 3$,
\[
t_m = c^{tr}(P_m), \quad p_m = c_0(P_m), \quad \hbox{and}  \quad l_m = c_1(P_m).
\]
Is it true that for $t_m$ and $p_m$, equality is attained only for affine regular $m$-gons, and for $l_m$, where $m \neq 4$, only for regular $m$-gons?
\end{problem}

\section{acknowledgements}
The support of the J\'anos Bolyai Research Scholarship of the Hungarian Academy of Sciences is gratefully acknowledged.
The authors are indebted to Endre Makai, Jr. for pointing out an error in an earlier version of the proof of Theorem~\ref{thm:translate},
and to an unknown referee for many helpful remarks.

\end{document}